\documentclass[10pt]{amsart}

 \usepackage{amssymb,amsmath,amsthm} 
 \usepackage{amsfonts}
\usepackage[mathscr]{eucal}
\usepackage{tikz}
\usepackage{amscd}
\usepackage{bbm}

\usepackage{verbatim}

\usepackage{graphics}

\newtheorem{theorem}{Theorem}[section]
\newtheorem{lemma}[theorem]{Lemma}

\newtheorem{corollary}[theorem]{Corollary}
\newtheorem{proposition}[theorem]{Proposition}

\newtheorem{remark}{Remark}

\theoremstyle{definition}

\newcommand{\R}{\mathbb{R}}

\newcommand{\Z}{\mathbb{Z}}

\newcommand{\N}{\mathbb{N}}

\newcommand{\ve}{\varepsilon}

\newcommand{\flr}[1]{\left\lfloor #1 \right\rfloor}
\newcommand{\cel}[1]{\left\lceil #1 \right\rceil}
\renewcommand{\mod}[1]{{\ifmmode\text{\rm\ (mod~$#1$)}\else\discretionary{}{}{\hbox{ }}\rm(mod~$#1$)\fi}}

\begin{document}

\title[Infinite symmetric ergodic index ] {Infinite symmetric ergodic index and related examples in infinite measure}

\author[Isaac Loh]{Isaac Loh}

\address{ }

\email{isaaccloh@gmail.com}

\author[C. E. Silva]{Cesar E. Silva}
\address{Department of Mathematics and Statistics, Williams College, 
\\ Williamstown, MA 01267} 
\email{csilva@williams.edu}

\author[Athiwaratkun]{Ben Athiwaratkun}
\address[Ben Athiwaratkun]{}
\email{ben.athiwaratkun@gmail.com}


\begin{abstract}
We define an infinite measure-preserving 
 transformation to have  infinite symmetric ergodic index if all finite Cartesian products of the transformation and  its inverse
 are ergodic, and show that infinite symmetric ergodic index does not imply that all products of powers are conservative, so does not  imply power weak mixing. We provide a sufficient condition for $k$-fold and infinite symmetric ergodic index and use it to answer a question on the relationship between product conservativity and product ergodicity. 
 We also show that a class of rank-one transformations that have infinite symmetric ergodic index are not power weakly mixing, 
 and precisely characterize a class of power weak transformations that generalizes existing examples.  
 \end{abstract}

\subjclass[2010]{Primary 37A40, 37A25; Secondary 28D05}

\keywords{Ergodic,  infinite measure, power weak mixing, infinite ergodic index}

\maketitle 


\section{Introduction}

In \cite{KaPa63}, Kakutani and Parry constructed infinite measure-preserving Markov shifts   such 
that all their finite  Cartesian products are ergodic and called this property
{\bf infinite ergodic index}. They proved that ergodic Cartesian square does not imply infinite ergodic index as it does in the finite measure-preserving case. In \cite{AdFrSi01}, Adams, Friedman and the second-named author 
constructed a rank-one transformation $T$ that has infinite ergodic index but 
such that $T\times T^2$ is not conservative, hence not ergodic. A  transformation such that  
 all finite Cartesian products of all its nonzero powers are 
 ergodic  is called {\bf power weakly mixing} \cite{DGMS99}, and if all finite Cartesian products of all its  powers are conservative it is said to be {\bf power conservative}. As we only consider 
 invertible transformations on nonatomic spaces,  under these assumptions, if a transformation is ergodic, then it is conservative (see e.g. \cite[3.9.1]{Si08}). It follows from \cite{AdFrSi01} that infinite ergodic index does not imply power conservative, so it does not imply 
power weak mixing.  In \cite{AdFrSi97}, Adams, Friedman and the second-named author 
modified the classic Chacon transformation by adding at each stage of the rank-one construction 
enough new intervals so that the measure  of a column at each stage   is twice    the measure of the  previous column. This results in an infinite measure-preserving transformation which is shown in \cite{AdFrSi97} to have infinite ergodic index, and that has been called the infinite Chacon transformation. It was later shown that the infinite Chacon is not power weakly mixing \cite{GrHiWa03} (in fact, not power conservative). 
Soon after \cite{AdFrSi01}, Bergelson asked if there existed a transformation $T$ that has
infinite ergodic index but such that $T\times T^{-1}$ is not ergodic. In \cite{CFKLPS}, Clancy et al. construct rank-one transformations $T$ such that $T\times T$ is ergodic but $T\times T^{-1}$ is not
ergodic; it is  also shown in \cite{CFKLPS} that for a rank-one $T$, the product  $T\times T^{-1}$ must always be conservative. The example with $T\times T^{-1}$  not ergodic was later generalized to 
more general group actions in Danilenko \cite{Da16}, but the full Bergelson question remains open.

We define a transformation $T$ to have {\bf infinite symmetric ergodic index}
if $T\times T^{-1}$ has infinite ergodic index. We show that infinite symmetric ergodic index does not imply power conservativity, so does not imply power weak mixing, by demonstrating in Section~\ref{S:Chacon-iei} that a condition, shared by the infinite Chacon transformation and several other existing examples with infinite ergodic index, implies
infinite symmetric ergodic index. A sufficient condition for finite symmetric ergodic index is also given in Proposition \ref{P:chaconergcond}, and is used in
 Theorem~\ref{T:asymmprods} to answer Question 1 of Danilenko \cite{Da16}. In Section~\ref{S:Chacon-npc}, we give a condition guaranteeing that a rank-one $T$ is not power conservative and explicitly bound the $k$ for which $T\times T^2 \times \cdots \times T^k$ is conservative. We show that a large class of infinite Chacon transformations meet this condition and  consequentially are not power weakly mixing, although they have infinite symmetric ergodic index. 
 In Section~\ref{S:Chacon-pwm} we show that power weak mixing holds in a class of transformations containing the primary example of  \cite{DGMS99} (and also the examples in \cite{AM}), and the bounded-recurrence example of \cite{GrHiWa03}. For terms not defined here the reader may refer for \cite{Si08} or \cite{CFKLPS}. 
 
\subsection{Acknowledgements}
We would like to thank partial support provided by the National Science Foundation REU Grant  
DMS-0850577 
and DMS - 1347804 and the Bronfman Science Center of Williams College.

\section{Preliminaries}

\subsection{Rank-One Cutting and Stacking}

Our main constructions will be obtained through rank-one cutting and stacking. We define a \textbf{Rokhlin column} $C$ to be a ordered and finite collection of \textbf{levels}, which are intervals in $\R$ of the same measure. A column $C = \{I_j\}$ is associated with a map taking every point, except for those in the topmost level, to the point directly above it in the above level. The levels of the column are at heights $0$ through $h$, where $h$ is called the \textbf{height} of the column (note that the column contains $h+1$ levels). The rank-one method involves iteratively constructing the columns and building an associated transformation as follows: 
\begin{enumerate}
	\item Take the first column to be the unit interval.
	\item To build $C_{n+1}$ from $C_n$, cut $C_n$ into $r_n \ge 2$ subcolumns of equal width and add $s_{n,k}$ new levels (\textbf{spacers}) above the $k$th subcolumn for $k \in [r_n-1]$. The spacers are levels taken from $\R$ that are disjoint from the levels of $C_n$. Then stack every subcolumn under the subcolumn to its right. 
	\item Build $C_n$ such that $X = \bigcup_{n \in \N} C_n$.
\end{enumerate}
There are many means of characterizing rank-one transformations but the one we use is the notation of \textbf{descendants}. We say that a level $I$ in $C_n$ splits into descendants in $C_{n+1}$, which have heights indexed by $D(I,n+1)$. If $I$ was at height $i$ in $C_n$, then it is easy to see that $D(I,n+1) = i + H_n$, where $H_n$ is the \textbf{height set} defined to be 
\[
H_n = \{0\} \cup \left\{ \sum_{\ell=0}^k (h_n + s_{n,k} +1): \, 0 \le k < r_n - 2 \right\}. 
\]
It is then clear that $h_{n+1} = h_n + s_{n,r_n-1} + \max H_n$. By an inductive argument, $D(I,n+j) = H_{n+j-1} + H_{n+j-2} + \cdots + H_n + i$, where by abuse of notation we use $+$ to refer to the sum of sets (i.e.\ $A + B = \{a+b: \, a \in A, b \in B\}$) and shifts of sets by integer addition.

\subsection{Product Ergodicity and Conservativity}

Lemma \ref{L:semiringce2} provides distinct necessary and sufficient conditions for ergodicity of product transformations. Its proof is standard (see e.g.\ Lemma 2.1 in \cite{CFKLPS}). 

\begin{lemma}\label{L:semiringce2}
	Let $T_0,\ldots,T_{k-1}$ be rank-one measure preserving transformations on measure spaces $X_0,\ldots,X_{k-1}$ and let $\alpha_0,\ldots,\alpha_{k-1}$ be nonzero integers. Let $T:=T_0 \times \cdots \times T_{k-1}$, and $X:=X_0\times \cdots \times X_{k-1}$. Let $\mathcal{D}$ be the sufficient semiring of rectangles of the form $R_0 \times \cdots \times R_{k-1}$, where $R_\ell$ is a level of some column of $T_\ell$. If $T$ is conservative ergodic, then for all $A,B \in \mathcal{D}$, we have 
	\[A\subset \bigcup_{n} T^n B \mod{\mu}.\]
Furthermore, $T$ is ergodic if for every $A,B \in \mathcal{D}$ of the form $A = I_0 \times \cdots \times I_{k-1}$ for $I_\ell$ the base of column $C_{\ell,i}$, $i\in \N$, and $B = T_0^{b_0} I_0 \times \cdots \times T_{k-1}^{b_{k-1}} I_{k-1}$, where  $0\le b_\ell < h_{\ell,i}$ for all $\ell$, $1\le \ell \le k-1$, there exists a one-to-one map $\tau: A \rightarrow B$ satisfying $\tau(x) \in \left\{ T^nx: \, n \in \Z \right\}$ for all $x \in A$ and:
\begin{align*}
&\text{ and }\mu(D(\tau)) \ge \beta(b_0,\ldots,b_{k-1}) \mu(A),\\
&\text{ and }\frac{d\mu\circ \tau}{d\mu}(v) \ge \delta(b_0,\ldots,b_{k-1}) \text{ for all } v\in D(\gamma).
\end{align*}
where $\delta(b_0,\ldots,b_{k-1}),\, \beta(b_0,\ldots,b_{k-1})>0$ are positive functions of the heights of sides of $B$, and $D$ and $R$ denote the domain and range of a map, respectively. 
\end{lemma}

Lemma \ref{L:cedescendants2} provides necessary conditions for the ergodicity of products of rank-ones, and follows from Lemma \ref{L:semiringce2} by argument of \cite{CFKLPS}, Lemma 2.2. We denote the descendant set of the $\ell^\text{th}$ indexed transformation $T_\ell$ by $D_\ell$.

\begin{lemma}\label{L:cedescendants2} 
	Let $T:=T_0^{\alpha_0} \times  \cdots \times T_{k-1}^{\alpha_{k-1}}$ be a product of (nonzero) integer powers of rank-one transformations in the product space $X_0\times \cdots \times X_{k-1}$.
	Fix $i,k\in\N$. Let $I_\ell$ be the base of column $C_{i,\ell}$ and $A = I_0 \times \cdots \times I_{k-1}$ be the product of $k$ such base levels, and $B = T^{b_0} I_0 \times T^{b_1} I_1 \times \cdots \times T^{b_{k-1}} I_{k-1}$, where $0\le b_\ell<h_{i,\ell}$ for $0\le \ell \le k-1$. 
	Then $T$ is conservative ergodic only if for every $\epsilon > 0$ and every choice of $i$ and $b_{0},\ldots,b_{k-1}$, there is a natural number $j>i$ such that for at least $(1 - \ve)\prod_{\ell = 0}^{k-1} \big|D_\ell(I_\ell,j)\big|$ tuples of descendants $(a_0, \ldots, a_{k-1}) \in D_0(I_0, j)\times \cdots \times D_{k-1}(I_{k-1},j)$ we have $a_\ell = d_\ell +b_\ell + \alpha_\ell n$ for $\ell = 0,\ldots,k-1$ for some tuple $(d_0,\ldots,d_{k-1}) \in D_0(I_0, j)\times \cdots \times D_{k-1}(I_{k-1},j)$ and $n\in \Z$. 
	
\end{lemma}

We now state these results as conditions on elements of products of the descendant sets:

\begin{proposition} \label{P:arbproduct} For rank-one transformations $T_0,\ldots,T_{k-1}$ and nonzero integers $\alpha_0,\ldots,\alpha_{k-1}$, $T:=T_0^{\alpha_0} \times \cdots \times T_{k-1}^{\alpha_{k-1}}$ is conservative ergodic only if for every $i\in \N$, $\ve>0$, and $b$ $k$-tuple $(b_0,\ldots,b_{k-1}) \in \prod_{\ell = 0}^{k-1} \{0,\ldots,h_{i,\ell}-1\}$, there is a natural number $j>i$ such that for at least $(1-\ve)\prod_{\ell = 0}^{k-1} |D_\ell(I_\ell,j)|$ $k$-tuples of descendants of the form $(a_0,\ldots,a_{k-1}) \in D_0(I_0,j)\times \cdots \times D_{k-1} (I_{k-1},j)$, we have corresponding $k$-tuples $(d_0,\ldots,d_{k-1}) \in D_0(I_0,j)\times \cdots \times D_{k-1} (I_{k-1},j)$ such that 
\begin{align} \label{E:ergeq}
\frac{a_0 - d_0 - b_0}{\alpha_0}= \frac{a_\ell - d_\ell - b_\ell}{\alpha_\ell} \in \Z \setminus\{0\}
\end{align}
for each $\ell = 0,\ldots,k-1$. 

$T$ is ergodic if there are constants $\beta(b_0,\ldots, b_{k-1})$ for all $b$ $k$-tuples such that the following \textbf{product ergodicity condition} holds:
There exists some $j \ge i$ such that to at least some fraction $\beta(b_0,\ldots , b_{k-1})$ of the $j$-stage descendant tuples $(a_0,\ldots,a_{k-1})$ with a complementary descendant $k$-tuple meeting \eqref{E:ergeq}, we can associate a unique such $d$ $k$-tuple. 
\end{proposition}
\begin{proof}
Necessity of \eqref{E:ergeq} follows from Lemma \ref{L:cedescendants2}. For sufficiency of the second, write $A = I_0 \times \cdots \times I_{k-1}$ and $B = T_0^{b_0}I_0 \times \cdots \times T_{k-1}^{b_{k-1}} I_{k-1}$. Let $F(j)$ denote all of the descendant tuples in $ D_0(I_0,j)\times \cdots \times D_{k-1} (I_{k-1},j)$ which are matched with unique complementary tuples $(d_0,\ldots,d_{k-1})$ meeting the stated conditions. By supposition, we can pick a $j$ such that $|F(j)|\ge \beta(b_0,\ldots , b_{k-1})\prod_{\ell = 0}^{k-1} |D_\ell(I_\ell,j)|$. Set $n = -\frac{a_0 - d_0 - b_0}{\alpha_0}$. Then for each tuple $(a_0,\ldots,a_{k-1}) \in F(j)$, 
\begin{align*}
T^n \left(T_0^{a_0} J_0 \times \cdots \times T_{k-1}^{a_{k-1}} J_{k-1}\right) = T_0^{d_0 + b_0}J_0 \times \cdots \times T_{k-1}^{d_{k-1}+ b_{k-1}} J_{k-1} \subset B,
\end{align*}
for some $n \in \Z$, where $(d_0,\ldots,d_{k-1}) \in D_0(I_0,j)\times \cdots \times D_{k-1} (I_{k-1},j)$. Because the correspondence of rectangles in $A$ indexed by $F(j)$ to covered rectangles in $B$ is one-to-one and measure preserving, we can infer by Lemma \ref{L:semiringce2} (with $\delta = 1$) that $T$ is ergodic. 
\end{proof}

The following result on the conservativity of products of rank-one transformations is \cite[Proposition 4.2]{CFKLPS}. For $k\in\N$, let $I^{(k)} \equiv \underbrace{I \times \cdots \times I }_{ k \text{ times}}$ and define $T^{(k)}$ similarly.

\begin{proposition} \label{P:conservprod2}
	Let $T$ be a rank-one transformation on a measure space $Y$, let $k\in\N$ and let
	$(\alpha_0,\ldots,\alpha_{k-1})$ be a $k$-tuple of nonzero integers, and $X$ the product 
	of $k$-fold copies of $Y$. Then the product transformation  $S:=T^{\alpha_0}\times \cdots \times T^{\alpha_{k-1}}$ on $X$ is conservative if and only if for every  $A = I^{(k)}$, where $I$ is the base of column $C_i$, for every $\ve > 0$ there is $j$ such that at for at least $(1 - \ve)|D(I,j)|^k$ of the $k$-tuples $(a_0,\ldots,a_{k-1}) \in D(I,j)^k$, there exist complementary $k$-tuples $(d_0,\ldots,a_{k-1}) \in D(I,j)^k$ satisfying $\frac{a_0-d_0}{\alpha_0} = \frac{a_\ell- d_\ell}{\alpha_\ell} \in \Z\setminus \{0\}$ for $\ell = 1,\ldots,k-1$. 
	\end{proposition}


\section{Infinite Symmetric  Ergodic Index on a General Class of Rank-One}
\label{S:Chacon-iei}

In this section, we consider a fairly broad class of transformations having infinite ergodic index. These incorporate subsets of a nice additive form into height sets $H_n$ infinitely often. In these subsets, elements will be spaced $z_n$ and $z_n+1$ apart in such a way that, under fairly light conditions, $T^{(k)}$ is ergodic for $k \in \N \cup \{\infty\}$.  

Let $I$ be a level of column $C_i$. 
As a general technique for rank-one transformations for elements $a_\ell$ and $d_\ell$ in $D(I,j),\, j \ge i$, we write the standard descendant sum decomposition as $a_\ell = \sum_{q = i}^{j - 1} a_{\ell,q}$ and $d_\ell = \sum_{q = i}^{j - 1} d_{\ell,q}$, where $a_{\ell,q}$ and $d_{\ell,q}$ are elements of the height set $H_q$. For any $x \in I \times \cdots \times I$, we also write the expansion of $x$ as $(x_i,x_{i+1},\ldots )$, where $x_q \in H_q^k$ is the tuple of height set elements corresponding to the subrectangle containing $x$. Throughout, we will let $M = \min \{ m : \, n_m\ge i\}$.

\begin{lemma} \label{L:bound}
Let $T$ be a rank-one transformation, $I$ a level of $C_i$, and $(b_1,\ldots , b_{k-1})$ positive integers. Furthermore fix a sequence $(\delta_{n_m})_{m \in \N}$ such that $\sum_{m} \delta_{n_m}^k \rightarrow \infty$. For each $m$, suppose that there are sets $E_{1,n_m}, \ldots, E_{k-1,n_m} \subset F_{n_m} \subset H_{n_m}^k$ such that $|E_{\ell,n_m}| \ge \delta_{n_m}^k |H_{n_m}|^k$, and $|F_{n_m}| \le D |E_{\ell,n_m}|$ for all $\ell$ and $m$, where $D$ is a positive constant. Let $\gamma = \sum_{\ell= 0}^{k-1} b_\ell$. Then there exists a positive number $K(b_0,\ldots , b_{k-1})$ such that: 
\begin{align} \label{E:longreq}
	\mu\Big( \Big\{ x \in I^{(k)} : \, &\text{the first } \gamma \text{ indices } m \text{ for which } x_{n_m} \in F_{n_m} \text{ have:} \\
	&x_{n_m} \in E_{1,n_m} \,\text{for }b_1 \text{ times, then}\nonumber\\
	&x_{n_m} \in E_{2,n_m} \,\text{for }b_2 \text{ times, then}\nonumber\\
	&\vdots \nonumber\\
	& x_{n_m} \in E_{k-1,n_m} \,\text{for }b_{k-1} \text{ times, } \Big\} \Big)\nonumber\\
	 >& K(b_1,\ldots , b_{k-1})\,  \mu(I^{(k)}). \nonumber
\end{align}
\end{lemma}
\begin{proof}
Enumerate the set of all possible sets of $\gamma$ distinct, nonnegative integers by $R_q, \, q \in \N$. Let $A_q = \{x: \, \text{ the first } \gamma \text{ indices } m \text{ for which } x_{n_m} \in F_{n_m} \text{ are in }M + R_q\}$. Then clearly $\bigsqcup_q A_q$ forms a partition for the subset of points of $I^{(k)}$ having $x_{n_m}$ for at least $\gamma$ indices $m$ in $F_{n_m}$. To see that this subset is just $I^{(k)}$, recall that $|F_{n_m}| \ge \delta_{n_m}^k |H_{n_m}|^k$ by assumption. By the construction of rank-one transformations, we can treat the elements $x$ as elements of a probability space over infinite sequences of events (with the events being subsets of height sets). For $m \in \N$, set $F_{n_m}$ as  events with 
\[\mathbb{P}(F_{n_m})\equiv \frac{\mu(x: \, x_{n_m} \in F_{n_m})}{\mu(I^{(k)}) }= \frac{| F_{n_m}|}{|H_{n_m}|^k} \ge \delta_{n_m}^k;
\]
then the $F_{n_m}$ are independent insofar as the probability of their finite or countable intersections is just the product of the probabilities. Because $\sum_{m} \mathbb{P}(F_{n_m}) \rightarrow \infty$, the second Borel-Cantelli Lemma implies that the probability that $F_{n_m}$ occurs infinitely often is $1$, which is to say that $\mu(x :\, x_{n_m} \in F_{n_m} \text{ i.o.}) = \mu(I^{(k)})$, which shows that $\bigsqcup_q A_q = I^{(k)} \mod{\mu}$, i.e.\ almost every point of $I^{(k)}$ lies in $A_q$ for some $q$.

Now consider any fixed $q$. Call the set in \eqref{E:longreq} $W$, and let $R_q|_a^b$ be the subset of $R_q$ ranging from its $a^\text{th}$ highest element to (including) its $b^\text{th}$. Then we have 
\begin{align*}
\frac{\mu (W\cap R_q)}{\mu(I^{(k)})} & = \frac{\prod_{p \in R_q|_0^{b_1-1}} |E_{1,n_{M+p}}| \cdots \prod_{p \in R_q|_{b_{k-2}}^{\gamma-1}} |E_{k-1,n_{M+p}}| }{\prod_{p \in R_q|_0^{b_1-1}} |H_{n_{M+p}}|^k \cdots \prod_{p \in R_q|_{b_{k-2}}^{\gamma-1}} |H_{n_{M+p}}|^k }\\
\frac{\mu(R_q)}{\mu(I^{(k)})} & = \frac{\prod_{p \in R_q|_0^{b_1-1}} |F_{n_{M+p}}| \cdots \prod_{p \in R_q|_{b_{k-2}}^{\gamma-1}} |F_{n_{M+p}}| }{\prod_{p \in R_q|_0^{b_1-1}} |H_{n_{M+p}}|^k \cdots \prod_{p \in R_q|_{b_{k-2}}^{\gamma-1}} |H_{n_{M+p}}|^k }\\
\frac{\mu(W \cap R_q)}{\mu(R_q)}& = \frac{\prod_{p \in R_q|_0^{b_1-1}} |E_{1,n_{M+p}}| \cdots \prod_{p \in R_q|_{b_{k-2}}^{\gamma-1}} |E_{k-1,n_{M+p}}| }{\prod_{p \in R_q|_0^{b_1-1}} |F_{n_{M+p}}| \cdots \prod_{p \in R_q|_{b_{k-2}}^{\gamma-1}} |F_{n_{M+p}}|}\\
& \ge D^{-\gamma}.
\end{align*}
Because the $R_q$ are disjoint, we have 
\begin{align*}
\mu(W) = \mu\left( \bigsqcup_q (W \cap R_q) \right) \ge D^{-\gamma} \mu\left( \bigsqcup_q R_q \right)  = D^{-\gamma} \mu\left( I^{(k)} \right). 
\end{align*}
We finish by taking $K(b_1,\ldots , b_{k-1}) = D^{-\gamma}$. 

\end{proof}

\begin{proposition} \label{P:chaconergcond}
Let $T$ be a rank-one transformation such that there is a sequence $(n_m)_{m\in \N}$ indexing positive integers $\left(z_{n_m}\right)_{m \in \N}$ such that 
\begin{align*}
&\left|S_{n_m}\left(z_{n_m}\right)\right| = \delta_{n_m} \left| H_{n_m}\right| \text{ and }\\
&\left|S_{n_m}\left(z_{n_m} + 1\right)\right| = \delta_{n_m} \left| H_{n_m}\right|
\end{align*}
for a sequence $\left(\delta_{n_m} \right)_{m \in\N}$ satisfying $\sum_{m \in \N} \delta_{n_m}^k \rightarrow \infty$. Then all $k$-fold products of $T$ of the form: 
\begin{align} \label{E:kfoldprod}
\underbrace{T \times\cdots \times T}_{t \text{ times}} \times \underbrace{T^{-1} \times \cdots \times T^{-1}}_{k-t \text{ times }},
\end{align}
with $0 \le t \le k$, are ergodic. 
\end{proposition}

\begin{proof}
Assume that $t \ge 1$ (this is without loss of generality, because the $k$-fold product of $T^{-1}$ is ergodic if and only if it is for $T$). By Proposition \ref{P:arbproduct}, the result follows if we show that for any $\ve > 0$, $(b_0,\ldots,b_{k-1}) \in \{0,\ldots,h_i -1 \}^{k}$, and $i \in \N$ (with $I$ the base of $C_i$), there exists a $j > i$ such that to at least $(1-\ve)\big| D(I,j) \big|^{k}$ of the $k$-tuples $(a_0,\ldots,a_{k-1}) \in D(I,j)^{k}$, we can associate unique $k$-tuples $(d_0,\ldots,d_{k-1}) \in D(I,j)^{k}$ satisfying $a_0 - d_0 - b_0 = a_\ell - d_\ell - b_\ell$ for $\ell = 1,\ldots,t-1$ and $a_0 - d_0 - b_0 = d_\ell - a_\ell + b_\ell$ for $\ell = t,\ldots,k-1$. Set a vector $(b_0,\ldots , b_k)$. For ease of notation, reorder the indices so that all $\ell \le t-1$ satisfying $b_0 \le b_\ell$ fall in the range $0,\ldots , s-1$, for $s \le t$. Then we may assume by adding $b_0$ to all equations that $b_0 = 0$, $b_\ell \ge 0 $ for all $\ell$ satisfying $1 \le \ell \le s-1$ or $\ell \ge t$, and $b_\ell < 0 $ for all other $\ell$. 

Let $M=\min\{m: \, n_m \ge i\}$ and $M'$ be high enough such that there exist tuples $(a_0,\ldots , a_{k-1}) \in D(I,n_{M'})^{k-1}$ satisfying the following properties $[1]_\ell$ through $[3]_\ell$; note that the indices $m$ satisfying these $k$ conditions are necessarily distinct. 
\begin{align} \label{E:tupledef}
\text{For all }&\ell \in \{1,\ldots, s-1\}:\\
[1]_\ell: \, &a_{\ell',n_m} \in S_{n_m}(z_{n_m}) \text{ for all }\ell' \neq  \ell \text{ and } a_{\ell,n_m} \in S_{n_m}(z_{n_m} + 1) \nonumber\\
& \text{for at least } b_\ell \text{ indices } m \in \{M,\ldots , M'\}\nonumber\\
\nonumber\\
\text{For all }&\ell \in \{s,\ldots, t-1\}:\nonumber\\
[2]_\ell:\, &a_{\ell',n_m} \in S_{n_m}(z_{n_m}+1) \text{ for all }\ell'\neq \ell \text{ and } a_{\ell,n_m} \in S_{n_m}(z_{n_m})\nonumber\\
& \text{for at least } -b_\ell \text{ indices } m \in \{M,\ldots , M'\}
\nonumber\\
\nonumber\\
\text{For all }&\ell \in \{t,\ldots, k-1\}:\nonumber\\
[3]_\ell:\, &a_{\ell',n_m} \in S_{n_m}(z_{n_m}+1) \text{ for all }\ell' \neq \ell \text{ and } a_{\ell,n_m} \in S_{n_m}(z_{n_m}) - z_{n_m}\nonumber\\
& \text{for at least } b_\ell \text{ indices } m \in \{M,\ldots , M'\}.\nonumber
\end{align}
For such a tuple, we can uniquely assign a complementary tuple $(d_0,\ldots , d_k)$ meeting the product ergodicity condition of Proposition \ref{P:arbproduct}. For all $\ell \in \{1,\ldots, s-1\}$, do the following: for the first $b_\ell$ indices $m$ satisfying $[1]_\ell$, set $d_{\ell,n_m} = a_{\ell, {n_m}} - z_{n_m} - 1 $ and $d_{0,n_m} = a_{\ell, {n_m} } - z_{n_m}$. Repeat similarly for $\ell \in \{s,\ldots , t-1\}$ and $\ell \in \{t,\ldots , k-1\}$ but replace the differences according to index (e.g.\ when $a_{\ell,q}$ is chosen to be in $S_{q} (z_{q}) - z_{q}$, set $d_{\ell,q} = a_{\ell,q} + z_q$). Then, everywhere $d_{0,q}$ has been assigned but $d_{\ell,q}$ has not, choose $d_{\ell,q}$ such that $a_{0,q} - d_{0,q} = a_{\ell,q} - d_{\ell,q}$ (for $\ell \le t-1$) or $a_{0,q} - d_{0,q} = d_{0,q} - a_{\ell,q}$ (for $\ell \ge t$). Elsewhere, set $d_{\ell,q} = a_{\ell,q}$. 

Now we use Lemma \ref{L:bound} to show that there is a set $W\subset I^{(k)} $ of size $\mu(W) \ge K(b_1,\ldots, b_{k-1}) \mu(I^{(k)} )$ such that the assignment of $d$ rectangles to $a$ rectangles on $W$ is unique. In the statement of Lemma \ref{L:bound}, take $E_{\ell,n_m}$ to be the $k$-tuple in $H_{n_m}^k$ meeting the condition relevant for $\ell$ in \eqref{E:tupledef}. For instance, $E_{1,n_m}$ is the subset of $H_{n_m}^k$ of tuples with $a_{\ell',n_m} \in S_{n_m}(z_{n_m}) \text{ for all }\ell' \neq  \ell \text{ and } a_{\ell,n_m} \in S_{n_m}(z_{n_m} + 1)$. Also, take 
\[F_{n_m} = \Big(S_{n_m} (z_{n_m}) \, \cup \, S_{n_m} (z_{n_m} +1 ) \, \cup \, S_{n_m} (z_{n_m})  - z_{n_m}\, \cup  \, S_{n_m} (z_{n_m}+1)  - z_{n_m} -1 \Big)^k.
\]
Then $|F_{n_m}| \le 4^k |E_{\ell,n_m}|$, and $|E_{\ell,n_m}| = |S_{n_m}(z_{n_m})|^k = \delta_{n_m}^k |H_{n_m}|^k$. Lemma \ref{L:bound} implies we can use $K(b_1,\ldots , b_{k-1}) = 4^{-k \sum_{\ell= 1}^{k-1} b_\ell}$. 

Now suppose that $a=(a_0,\ldots, a_{k-1})^k$ and $a'= (a_0',\ldots , a_{k-1}')^k$ are assigned to the same $d$ rectangle. It must be the case that every time $a_q = (a_{0,q},\ldots , a_{k-1,q})$ is assigned a $d_q \neq a_q$, we have $d_q \in F_q$. If $a_q' = d_q$ in this stage, then we have $a_q' \in F_q$, implying by definition of $W$ that $a'$ meets conditions $[1]_\ell$ through $[3]_\ell$ in indices strictly below $q$. But everywhere that $a_p' \neq d_p$ we must have $a_p \in F_p$, and in indices $p < q$, $a$ does not meet conditions $[1]_\ell$ through $[3]_\ell$ (else, we would not have $a_q \neq d_q$). This contradicts that $a$ meets conditions $[1]_\ell\text{---}[3]_\ell$ in the first $\gamma$ indices in which it is in $F_p$. So $a_q \neq d_q \iff a_q' \neq d_q$, i.e.\ $a$ and $a'$ belong to the same partition block in Lemma \ref{L:bound}. Thus, it is straightforward to check by the definition of $W$ and the algorithm assigning $d$ rectangles to $a$ rectangles that $a = a'$. Proposition \ref{P:arbproduct} finishes, with $\beta(b_0,\ldots, b_{k-1})$ given by $K(b_1- b_0,\ldots, b_0 - b_s,\ldots, b_0 + b_t,\ldots )$. 

\end{proof}

\begin{corollary} \label{C:fixeddelta}
If $\delta_{n_m} \ge \delta > 0$ for all $m$, then $T$ has infinite symmetric ergodic index. 
\end{corollary}

In the case of infinite rank-one integer actions, Proposition \ref{P:chaconergcond} implies claim 1 of Theorem 0.1 in \cite{Da16} (consider $f = (0,1,\ldots , 0,1)$). Other well known examples, such as the main construction in \cite{AdSi16} and the infinite measure Chacon transformation with $2$ or more cuts, satisfy the conditions of Corollary \ref{C:fixeddelta} to obtain infinite ergodic index. To address the notion of mixing for rank-one transformations, we require the following lemma. Its proof is similar to that of Claim 6, \cite{Da16}. We say that $T$ is {\bf Koopman mixing} if for all sets $A,B$ of finite measure, $\lim_{n\to\infty}\mu(T^nA\cap B)=0$; in the finite measure preserving case this is equivalent to the spectral characterization of (strong) mixing, and in infinite measure it is also known as zero-type.

\begin{lemma} \label{L:mixing}
Let $T$ be a rank-one transformation with height sets $H_n = S_n \sqcup R_n$ chosen such that $|S_n| < \delta_n |H_n|$ with $\delta_n \rightarrow 0$ and $|H_n| \rightarrow \infty$, and such that if $x,y,z,z' \in H_n$ with $x \in R_n, \, x \neq y$ and $(z,z') \neq (x,y)$, then $|x - z - y  + z'| \ge 2 h_n$. If $T$ adds at least $\max\{ H_n\} + h_n$ spacers on the right subcolumn for each $n$, then it is Koopman mixing.
\end{lemma}
\begin{proof}
Let $F$ be a collection of levels in column $C_i$. For any $n \ge i$, $D(F,n+1)$ consists of $|H_n|$ copies of $D(F,n)$ indexed by $H_n$. We claim that for $m$ satisfying:
$
\max D(F,n) \le m \le \max D(F,n+1),
$
we have 
\begin{align}\label{E:ineqHn0}
\mu(T^m(F) \cap F) \le \max \{ |H_n|^{-1}, \delta_n \} \mu(F)
\end{align}
Note that such an $m$ ensures that a copy of $D(F,n)$ in $D(F,n+1)$ cannot intersect with itself. Observe that for such $m$ (by addition of the absorbing spacers on the right subcolumn)
\begin{align} \label{E:ineqHn}
\mu(T^m(F) \cap F) \le \mu(F) \frac{\left| (x,y) \in H_n^2: \, (D(F,n) + x + m )  \cap D(F,n) + y \neq \emptyset \right|}{|H_n|}.
\end{align}
If one of $\{x,y\}$ is in $R_n$, then the condition on $R_n$ implies that $(D(F,n) + z + m) \cap D(F,n) + z' = \emptyset$ for all $(z,z') \neq (x,y)$, so the right side of \eqref{E:ineqHn} can be bounded by $|H_n|^{-1} \mu(F)$. On the other hand, the maximal intersection between $S_n$-indexed copies of $D(F,n)$ has size $|D(F,n+1)|\frac{|S_n|}{|H_n|} \le \delta_n |D(F,n+1)|$. It follows that for such an $m$, $\big|(m+D(F,n+1))\cap D(F,n+1) \big| \le \max \{ |H_n|^{-1}, \delta_n \}|D(F,n+1)|$, so \eqref{E:ineqHn0} holds. As $\max \{ |H_n|^{-1}, \delta_n \} \rightarrow 0$, $T$ is mixing for collections of levels, whence mixing for arbitrary sets (see e.g.\ Theorem 9.6 in \cite{LoSi17}). 
\end{proof}

With the use of Proposition \ref{P:chaconergcond}, we are able to answer Question 1 of Danilenko in \cite{Da16}:

\begin{theorem} \label{T:asymmprods}
Fix any $k \ge 1$ and $p \ge k$ (including $p = \infty$). Then there exists a Koopman mixing rank-one transformation $T$ such that all $k$-fold products of the form \eqref{E:kfoldprod} are ergodic, $T$ has ergodic index $k$, and $T$ has conservative index $p$. 
\end{theorem} 
\begin{proof}
Let $p < \infty$ to start. 
In the previous construction, take $n_m$ to be the odd numbers. For every even $n$, take $H_n$ to be a set of $r_n$ elements satisfying $x,y,z,z' \in H_n \text{ with } (x,y) \neq (z,z') \implies |x-z-y+z'| \gg 2 h_n \ge 2 \max D(I_0,n)$ (i.e.\ if the difference of differences is nonzero, then it is large). 

For ease of notation, let $S_n = \big(S_n(z_n) \cup \ldots \cup S_n(z_{n}+1)-z_n-1 \big)$. For every odd $n$, choose $R_n$ and $H_n = R_n \sqcup S_n$ such that $x,y,z,z' \in H_n \text{ with } x \in R_n, (x,y) \neq (z,z') \, \implies |x-z-y+z'| \gg 2 h_n \ge 2 \max D(I_0,n)$. This is a stronger version of condition (2-1) in \cite{Da16}, and similar to the restriction discussed in Remark 1, \cite{CFKLPS}. Add $\max \{H_n\}+ h_n$ spacers on the rightmost subcolumn for every $n$, and choose $\delta_n$ such that $\sum_{n \text{ odd}} \delta_n^k \rightarrow \infty$ but $\sum_{n \text{ odd}} \delta_n^{k+1} < \infty$. Lastly, pick $(r_n)$ such that $\prod_{i=1}^\infty \left( 1 - \frac{1}{r_i^{p-1}} \right) = 0$ but $\prod_{i =1}^\infty \left( 1 - \frac{1}{r_i^{p}} \right)> 0$. Throughout, we let $\text{diag}_k(A)$ denote the set of $k$-tuples $\{(x,\ldots , x): \, x \in A\}$. 

First, we argue that $T^{(p+1)}$ is not conservative. By selection of $\delta_n$ and $r_n$, there exists an $i$ so high and $I$ a level of $C_i$ such that 
\begin{align*}
\Big|(&H_{i}^{p+1} \setminus S_i^{p+1} + H_{i+1}^{p+1} + H_{i+2}^{p+1} \setminus S_{i+2}^{p+1} +\cdots )\\
 &\cap (H_i^{p+1}  \setminus 
\text{diag}_{p+1}(H_i) + \cdots + H_{j-1}^{p+1}\setminus 
\text{diag}_{p+1}(H_{j-1})  \Big|\\
&> \ve |D(I,j)|^{p+1}
\end{align*}
for all $j \ge i$. Assume that $(a_0,\ldots , a_{p})$ is an element of the above set in $D(I,j)^{p+1}$. If there is a complementary descendant tuple $(d_0,\ldots , d_{p})$ such that $a_0 - d_0 = \cdots = a_p - d_p \neq 0$, there exists a highest $q \le j-1$ such that $a_{\ell,q} - d_{\ell,q}$ are not all $0$. Fix an $\ell$ such that this is true; by assumption, there is another $\ell'$ such that $a_{\ell,q} \neq a_{\ell',q}$. Since $d_{\ell,q} \neq a_{\ell,q}$, the inequality 
\[
|a_{\ell,q} - d_{\ell,q} - a_{\ell',q} + d_{\ell',q} | \gg 2 \max D(I_0,q) \ge 2 \max D(I,q)
\]
holds, and it is impossible to have $a_{\ell,q}- d_{\ell,q} = a_{\ell',q} - d_{\ell',q}$, as was the case of Lemma 3.2 in \cite{CFKLPS}. Since $\ve$ is a fixed constant, Proposition \ref{P:conservprod2} implies that $T^{(p+1)}$ is not conservative. However, $T^{(p)}$ is conservative by Theorem 3.2 in \cite{LoSi17}.

Now we show that $T^{k+1}$ is not ergodic. A similar result is proven for Theorem 0.1 in \cite{Da16}; however, we sketch a different argument along the lines of the one already given. Let $(a_0,\ldots , a_{k}) \in D(I,j)^{k+1}$ be such that that in every odd $q$ we have $(a_{0,q},\ldots , a_{k,q}) \in H_{q}^{k+1} \setminus S_{q}^{k+1}$; we know that $i$ can be chosen such that some positive fraction exceeding $\ve > 0$ of the elements in $D(I,j)^k$ satisfy this. 

Fix a vector $(b_0,\ldots , b_k)$ such that not all of its elements are equal, and all of its elements are small. Then for the equality $a_0 - d_0 - b_0 = \cdots = a_k - d_k - b_k$ to hold, there must be a highest index $q$ such that $a_{0,q} - d_{0,q}, \ldots , a_{k,q} - d_{k,q}$ are not all equal. Then if $q$ is odd, the fact that $a_{\ell,q} \in H_{q}^{k+1} \setminus S_{q}^{k+1}$ for some index $\ell$ and $a_{\ell,q} - d_{\ell,q} \neq a_{\ell',q} - d_{\ell',q}$ for some $\ell'$ implies that $|a_{\ell,q} - d_{\ell,q} - a_{\ell',q} + d_{\ell',q} | \gg 2 \max D(I,q)$. As the $b_\ell$'s are small, this contradicts that $a_\ell - d_\ell - b_\ell = a_{\ell'} - d_{\ell'} - b_{\ell'}$. The case where $q$ is even is handled similarly, and Proposition \ref{P:arbproduct} concludes. 

Finally, note that $T$ is Koopman mixing by Lemma \ref{L:mixing}.

The case where $p = \infty$ is handled easily by choosing even height sets to have bounded cuts. Hence, $T$ is partially rigid and has infinite conservative index (as in Proposition 5.1 of \cite{LoSi17}). 
\end{proof}

\begin{remark} \normalfont
It is easy to see that not all of the transformations covered by Proposition \ref{P:chaconergcond} are isomorphic to their inverse. Consider the infinite-measure Chacon transformation with $3$-cuts that uses height sets $H_n = \{0, h_n, 2h_n +1 \}$ with $h_0 = 1$, $h_{n+1} = 6h_n +2$ (so $3h_n + 1$ spacers are added on the last subcolumn). Let $I$ be a level of column $C_i, i \ge 1$. For $n \ge i$, it is clear that $I \cap T^{h_n +1} I \cap T^{2h_n + 1} I $ is a union of  levels in the third subcolumn of $C_n$. However, the heights of the sublevels of $I$ in the third subcolumn of $C_n$ are at heights $2h_n +1+ D(I,n)$, and the heights of the sublevels of $T^{h_n + 1} I$ in the third subcolumn are at heights $2h_n + 2 + D(I,n)$. It is not difficult to show that $D(I,n)$ never contains two adjacent levels, so $|2h_n + 1 + D(I,n) \cap 2h_n + 2 + D(I,n)| = |D(I,n) \cap 1 + D(I,n)| = 0$. Thus, for the infinite Chacon: 
\begin{align*}
\mu \left( I \cap T^{-h_n} I \cap T^{-2h_n - 1} I \right) = \mu \left( I \cap T^{h_n + 1} I \cap T^{2h_n +1 } I \right) = 0.
\end{align*}
On the other hand, observe that for any positive finite-measure set $A \subset X$, we have 
\begin{align*}
\mu \left( A \cap T^{h_n} I \cap T^{2h_n + 1} A \right) \rightarrow \frac{1}{3} \mu(A).
\end{align*}
Thus, $T$ is not isomorphic to its inverse. We note that the fact that the infinite Chacon transformation is not isomorphic to its inverse was shown by de la Rue, Janvresse and Roy \cite{JRR} by different methods. More recently, Danilenko \cite{Da16b} has constructed a different class of infinite Chacon transformations and shown they are not isomorphic to their inverses.
\end{remark}

\begin{remark}  \normalfont
We have mentioned that Bergelson's question of whether infinite ergodic index implies infinite symmetric ergodic index remains open. It is was shown in \cite{CFKLPS} that when $T$ is rank-one, $T\times T^{-1}$ is always conservative, though this is not the case in general. One can ask 
whether infinite conservative index implies infinite symmetric conservative index, even in the rank-one class. 
\end{remark}



\section{Infinite Chac\'{o}n Type Transformations are Not Power Conservative}
\label{S:Chacon-npc}

The canonical $3$-cut Chac\'{o}n Type Transformation was first extended to infinite measure in \cite{AdFrSi97}. The authors showed that the transformation had infinite ergodic index. This result accords with the observation, made below, that all $T$ in a broader class defined by \eqref{E:chacondef} have $(T\times T^{-1})^{(k)}$ ergodic for all $k \in \N$. In \cite{GrHiWa03}, it was shown that the transformation defined in \cite{AdFrSi97} was not $7$-power conservative, hence not power weakly mixing. Infinite Chacon-type transformations produced using $(C,F)$ construction techniques were studied in \cite{Da04}, wherein by Theorem 0.3 a $2$-cut class with fast growth in $h_n$ was shown to be not power conservative. More recently, this transformation was the object of study in \cite{JRR}, where it shown to have trivial centralizer and no nontrivial factors. 

In Lemma \ref{L:dminusdap}, we give a condition on the height sets of $T$ which forces $T$ to be not power conservative, and explicitly bound the $\kappa$ such that $T \times \cdots \times T^{\kappa +1}$ is conservative. Corollary \ref{C:chaconnotpc} provides two simpler and sufficient conditions for the former result. All infinite Chac\'{o}n type transformations meet these conditions, and in fact a larger class of transformations which add large spacer stacks on their rightmost subcolumns do as well. 
 	
We define an \textbf{infinite Chac\'{o}n type transformation} as one which uses height sets $H_n,\, n \ge 0$ of the form  
\begin{align} \label{E:chacondef}
 H_n = \Big\{0,\mathbbm{1}_{\{1 \ge q\}}+h_n,\ldots,\mathbbm{1}_{\{t-1 \ge q\}}+(t-1) h_n\Big\},
 \end{align}
where $\mathbbm{1}$ is the indicator function and $q$ is some fixed element of $\{1,\ldots,t-1\}$. Thus, the selection of $q$ indicates that we add one spacer on the $q-1^\text{th}$ subcolumn of $C_n$, where the subcolumns are $0$-indexed. The parameter $t > 1$ is fixed, implying that we make $t-1$ cuts at each stage in order to get $t$ subcolumns. We stipulate that $h_{n+1} = m_{1} h_{n} + m_0$ for some absolutely constant real numbers $m_1,m_0$, where $m_1 \ge 2t$ and $m_0 \ge 1$; thus, we add $(m_1 - t) h_{n} + m_0 - 1$ spacers on the last subcolumn of $C_n$ in order to get $C_{n+1}$. Corollary \ref{C:fixeddelta} implies that any $T$ defined by \eqref{E:chacondef} forces $T\times T^{-1}$ to have infinite ergodic index. 

In the following lemmas, whenever we use the descendant set $D(I,n)$, we assume that $I$ is a level drawn from a column $C_i$, where $i \le n$. 
Also, recall that $H_n$ denotes the $n$th height set and $h_n$ is the height of the highest level in $C_n$. 
\begin{lemma} \label{L:dminusdap}
Let $T$ be a rank-one transformation with the property that 
\begin{align}\label{E:hineq}
\liminf_{n \rightarrow \infty} \frac{h_n - 2 \max \sum_{k=0}^{n-1} \max H_k}{\sum_{k=0}^n \max H_k} = \liminf_{n \rightarrow \infty} \frac{h_n - 2 \max D(I_0,n)}{\max D(I_0,n)} > \kappa^{-1}
\end{align}
for some positive $\kappa$ (equivalently, that the condition holds when $I_0$ is replaced with some fixed level $I_i$). Then $T$ is not power conservative, hence not power weakly mixing. Specifically, $T \times \cdots \times T^{\kappa + 1}$ is not conservative. 
\end{lemma}
\begin{proof}
Assume that there exist integers $N$ and $\kappa$ such that 
\begin{align}
\sup_{n \ge N} \frac{\max D(I_N,n+1)}{h_n - 2\max D(I_N,n)} = \sup_{n \ge N} \frac{\sum_{k=N}^n \max H_k}{h_n - 2 \max \sum_{k=N}^{n-1} \max H_k}&  \nonumber\\
\le \sup_{n \ge N} \frac{\sum_{k=0}^n \max H_k}{h_n - 2 \max \sum_{k=0}^{n-1} \max H_k} &< \kappa, \label{E:dminusdineq}
\end{align}
where $I_N$ is the base level of $C_N$. 

We claim by induction that $D(I_N,n) - D(I_N,n)$ contains no arithmetic progressions (APs) of the form $a-d,2(a-d),\ldots , (\kappa+1)(a-d)$ for some $a-d > 0$. The base case of $n = N$ is clear, as $D(I_N,N)$ is a singleton. For the inductive step, we can show that the existence of $a-d,\ldots , (\kappa+1)(a-d)$ is a contradiction by considering three cases. Note first that the smallest nonnegative element of $D(I_N,n+1) - D(I_N,n+1)$ that is not in $D(I_N,n) - D(I_N,n)$ is at least $h_n - \max D(I_N,n) > \max D(I_N,n)$ (by assumption). 
\begin{enumerate}
\item[\textbf{Case 1}:] If $(\kappa+1)(a-d) \in D(I_N,n) - D(I_N,n)$, then the entire AP occurs in $D(I_N,n) - D(I_N,n)$, which is impossible by hypothesis.
\item[\textbf{Case 2}:] If $ (a-d) \in D(I_N,n) - D(I_N,n)$ but $(\kappa + 1) (a-d) \not\in D(I_N,n) - D(I_N,n)$, then the AP bridges the difference between the largest element of $D(I_N,n) - D(I_N,n)$ and the smallest other (positive) element of $D(I_N,n+1)-D(I_N,n+1)$. Hence, the AP has common difference at least $h_n - \max D(I_N,n) - D(I_N,n)$, and by \eqref{E:dminusdineq} it cannot occur entirely within $D(I_N,n+1) - D(I_N,n+1)$. 
\item[\textbf{Case 3}:] If $(a-d) \not\in D(I_N,n) - D(I_N,n)$ then by \eqref{E:dminusdineq}:
\[ (\kappa+1) (a-d) \ge (\kappa+1)( h_n - \max D(I_N,n)) > \max D(I_N,n+1),\]
which is impossible.
\end{enumerate}
 Hence, for $n \ge N$, $D(I_N,n)$ excludes any APs of the form $(a-d),\ldots , (\kappa +1) (a-d)$ for $(a-d)$ positive, and by symmetry it has no such APs for $(a-d)$ negative. Then by Proposition \ref{P:conservprod2}, $T \times T^2 \times \cdots \times T^{\kappa + 1}$ cannot be conservative. 
\end{proof}

\begin{corollary}\label{C:simplernotpc}
If $\exists N$ such that 
\begin{enumerate}
\item For all $n \ge N$, $h_{n+1} \ge 2 h_n + 2 \max H_n + K$ for some fixed $K$ (i.e.\ the mass of the column is added as spacers on its last subcolumn)
\item $\frac{h_n}{\max H_{n+1}} \ge b > 0$ for some fixed $b$
\end{enumerate}
then $T$ is not power conservative.
\end{corollary}
\begin{proof}
Applying the first inequality to the left side of \eqref{E:hineq} implies
\begin{align*}
\frac{h_{n+1} - 2 \max \sum_{k=0}^{n} \max H_k}{\sum_{k=0}^{n+1} H_k} &\ge \frac{h_n + h_n + K  - 2\max \sum_{k=0}^{n-1} \max H_k}{\max H_{n+1} + \sum_{k=0}^{n} \max H_k} \\
& \ge \frac{b \max H_{n+1}+h_n + K - 2\max \sum_{k=0}^{n-1} \max H_k}{\max H_{n+1} + \sum_{k=0}^{n} \max H_k}.
\end{align*}
Applying the inequality iteratively to $n+2,\ldots$ shows that the left side of \eqref{E:hineq} is lower bounded by, say, $\frac{b}{2} > 0$ for sufficiently large $n$. Addition of a linear term in $n$ in the numerator is inconsequential because $\max H_{n}$ grows exponentially.
\end{proof}

Noting that Chac\'{o}n type transformations fulfill these conditions allows us to say the following:

\begin{corollary}\label{C:chaconnotpc}
All infinite Chac\'{o}n type transformations are not power conservative. 
\end{corollary} 
\begin{proof}
The first condition of Corollary \ref{C:simplernotpc} with $K = -1$ is obviously fulfilled by infinite Chac\'{o}n transformations, so all that is left is to show the second. But by construction,
\begin{align*}
\frac{h_n}{\max H_{n+1}} = \frac{h_n}{1 + (t-1) (m_1 h_n + m_0)} \rightarrow \frac{1}{(t-1) m_1} >0. 
\end{align*}
\end{proof}


\section{All $(t,q)$-type Chac\'{o}n Maps are Power Weakly Mixing}
\label{S:Chacon-pwm}

A rank-one $T$ is said to be $\mathbf{(t,q)}$-\textbf{Chac\'{o}n Type Transformation} if $t \ge 3$, $q \ge 1$, and at the $n^\text{th}$ stage we cut $C_n$ into $t$ subcolumns and distribute $q$ spacer blocks of height $h_n$ among subcolumns $C_{n,0}$ through $C_{n,t-2}$. We then add one spacer to the last, rightmost column $C_{n,t-1}$. For the purposes of the following section, we set $k = q + t$ to denote the number of segments of height $h_n$ contained in $C_{n+1}$. Note that $h_n = k h_{n-1} + 1$ and hence $h_n = \sum_{i=0}^n k^i$. The height sets $H_n$ of $T$ thus take the form 
\begin{align*}
H_n& = \left\{ 0 , \phi(1) h_n,\phi(2) h_n ,\ldots, \phi(t-1) h_n \right\},
\end{align*}
where $1 \le \phi(i) - \phi(i-1) \le 2$, and $\phi(t-1) = k-1$. For example, the main construction of \cite{DGMS99}, studied further in \cite{GrHiWa03}, is $(t,q)$ with $t = 4$ and $q = 1$, where one $h_n$-spacer block is added to the second subcolumn of $C_n$.

In contrast to the result of Proposition \ref{C:chaconnotpc}, we give a necessary and sufficient condition for the $(t,q)$-type Chacon transformation to be power weakly mixing. Throughout, we will let 
\[[n] = \{0,\ldots ,n\}.\]

\begin{lemma} \label{L:incompletediff} For any set $A=\big\{\phi(0),\phi(1),\ldots,\phi(t-1)\big\},\, t > 1$ where 
\begin{enumerate}
\item $\phi(0)=0$ and $\phi(t-1)=k-1$ for some integer $k$, and
\item $1\le \phi(i+1)-\phi(i)\le 2$ for $i \in \{0,\ldots,t-2\}$,
\end{enumerate}
if $\phi(i)-\phi(j)=1$ for any $i,j$ in $\{0,\ldots,t-1\}$ then 
\[\left[\left\lceil\frac{ k }{2}\right\rceil \right] \cup \Big\{z:\, 0\le z \le k-1 \text{ and } z \equiv k-1 \mod{2}\Big\} \subset A-A.
\]
\end{lemma}
\begin{proof}
Let $A$ be as specified; then there exist $i,j\in \{0,\ldots,t-1\}$ which satisfy $\phi(i)-\phi(j)=1$. Let $x\in \{1,\ldots,\phi(j)-1\}$ and suppose that $x\not\in A-A$. Then $\phi(i)-x$ and $\phi(j)-x$ cannot be elements of $A$. This presents a contradiction to the second condition on $A$, because then $\phi(\ell) - \phi(\ell-1) > 2$ for some element $\ell \in \{1,\ldots,j\}$ for which $\phi(\ell-1) < \phi(j) - x < \phi(i) - x < \phi(\ell)$. So $\{1,\ldots,\phi(j) - 1\}\subset A-A$. Also, trivially, we have $0\in A-A$ and $\phi(j),\phi(i)\in A-A$ because $0\in A$. Hence, $\big\{0,\ldots,\phi(i)\big\}\subset A-A$. A similar argument shows that $\big\{0,\ldots,k-1-\phi(j)\big\}\subset A-A$. Note that $\phi(j)=\phi(i)-1$ so $k-1-\phi(j)=k-\phi(i)$. Since
\[\max\big\{\phi(i),\, k-\phi(i)\big\}\ge \cel{\frac{k}{2}},\]
we have $\left[\left \lceil\frac{ k }{2}\right\rceil \right]\subset A-A$. 

Now we want to show that $A-A$ contains all of the elements of $\{0,\ldots,k-1\}$ with the same parity as $k-1$. Clearly $k-1\in A-A$. Suppose that $k-3\not\in A-A$. Then $2,\, k-3\not\in A$. But this implies that $1,k-2\in A$, which is a contradiction because $k-2-1=k-3$. So at least one of $1,2,k-3,k-2$ is in $A$, which gives us $k-3\in A$. To continue the argument, let $k$ and $\ell$ be the elements of $\{0,\ldots,t-1\}$ which give $\phi(k)-\phi(\ell)=k-3$. Suppose that $k-5\not \in A-A$; then $\phi(\ell)+2$, $\phi(k)-2$ are both not in $A$. But this implies that $\phi(\ell)+1$ and $\phi(k)-1$ are both in $A$, and we have $\phi(k)-1-(\phi(\ell)+1) = k-5$. A similar argument works for $k-7$, $k-9$, etc., which concludes. 
\end{proof}

\begin{corollary} \label{C:lowerhalf} We have 
\begin{equation}
\left[\left \lceil \frac{k}{2}  \right\rceil k^{n-1}\right] \subset (A-A)+k(A-A)+k^2(A-A)+...+k^{n-1}(A-A). \label{E:lowerhalf}
\end{equation}
\end{corollary}
\begin{proof}
Proceed by induction. The base case with $n=1$ is implied by Lemma \ref{L:incompletediff}. Suppose that \eqref{E:lowerhalf} holds for some $n\in \N$. Note that the set on the right hand set of (\ref{E:lowerhalf}) is symmetric about zero, so for any element $x\in (A-A)$ and any element $y\in \left[k^{n+1}\right]$ with $|y-k^n x| \le \left \lceil \frac{k}{2}  \right\rceil k^{n-1}$, we will have $y\in \big((A-A)+k(A-A)+k^2(A-A)+...+k^{n-1}(A-A)\big)+k^n(A-A)$. By Lemma \ref{L:incompletediff}, the set $\{k^n x:\, x\in A-A\}$ contains $\big\{0,k^n,...,\left \lceil \frac{k}{2}  \right\rceil k^n \big\}$, and we have shown that any integer that is contained in an closed interval of diameter $\left \lceil \frac{k}{2}  \right\rceil k^{n-1}$ centered around any element of this set is contained in the right side of (\ref{E:lowerhalf}).  Now $\frac{k^n}{2}\le \left \lceil \frac{k}{2}  \right\rceil k^{n-1}$, so we deduce that
\[\left[\left \lceil \frac{k}{2}  \right\rceil k^{n}\right] \subset (A-A)+k(A-A)+k^2(A-A)+...+k^{n-1}(A-A)+k^n(A-A),\]
which shows that (\ref{E:lowerhalf}) holds for the $n+1$ case. 
\end{proof}

\begin{corollary} \label{C:evenodd} All of the elements of $\left[k^n -1 \right]$ with the same parity as $k-1$ are contained in $(A-A)+k(A-A)+...+k^{n-1}(A-A)$, for every $n\in \N$.
\end{corollary}
\begin{proof} Again, proceed by induction. The base case is clear by Lemma \ref{L:incompletediff}. Suppose that the desired result holds true for some $n\in \N$. If $k$ is even, then $k-1$ is odd and by supposition we have $\{1,3,...,k^n-1\} \subset (A-A)+k(A-A)+...+k^{n-1}(A-A)$. By Lemma \ref{L:incompletediff} $A-A$ must contain the odds $\{1,...,k-1\}$; adding and subtracting elements from $\{1,3,...,k^n-1\}$ to the set $k^n(A-A)$, all of the odds from $1$ to $k^{n+1}-1$ must be contained in $(A-A)+k(A-A)+...+k^{n}(A-A)$, which completes the inductive step. The proof is similar when $k$ is odd, and is deduced by adding and subtracting the even numbers in $\{0,2,...,k^{n-1}-1\}$ to the even numbers in $k^n\cdot \{0,2,...,k-1\}\subset k^n (A-A)$. 
\end{proof}

\begin{lemma} \label{L:incompletediff2} Let $A$ be as specified in Lemma \ref{L:incompletediff} with the additional stipulation that $k \ge 3$, and let $B\subset \Z^+$ be a finite set with $B = \{\beta_1,...,\beta_v\}$. Then there exists some positive integers $n,m$ such that, for some element $\gamma$ with $k^m-\gamma \in (A-A)+k(A-A)+...+k^{m-1}(A-A)$, we have $k^n-\gamma \beta_q \in (A-A)+k(A-A)+...+k^{n-1}(A-A)$ for every $q \in \{1,...,v\}$. 
\end{lemma}
\begin{proof}
For brevity, set $D(n)=\sum_{i=0}^{n-1} k^i(A-A)$ (a sum set). Let $C = \left[k-1\right] \setminus (A-A)$, and let $g =|C|$. For every $n\in \N$, we will set
\[\lambda_n = \Big| \left[k^n-1\right]\setminus D(n) \Big|.\]
Clearly $D(1) = A-A$, so $r_1 = |C| = g$. For a natural number $n$, suppose that the value $\lambda_n$ is known; we wish to determine $r_{n+1}$. To do this, we will count the elements in $\left[k^{n+1}-1\right]$ which do not appear in $D(n+1)$. We can imagine the set $A-A$ as containing ``gaps", around which these elements are clustered. Corollary \ref{C:lowerhalf} implies that $\left[\left \lceil \frac{k}{2}  \right\rceil k^{n}\right]\subset D(n+1)$. Applying Corollary \ref{C:lowerhalf} to $D(n)$ also implies that the only elements in $[k^{n+1}-1]$ not appearing in $D(n+1)=D(n)+k^n(A-A)$ are more than $\left \lceil \frac{k}{2}  \right\rceil k^{n-1} \ge \frac{k^n}{2}$ away from any element of the form $xk^n$, where $x\in A-A$. Suppose that $x,y,z$ are elements of $[k-1]$ with $x<y<z$; suppose further that $y\not\in A-A$. Then by supposition (on the preconditions of Lemma \ref{L:incompletediff}), $x,z\in A-A$. Hence, by omitting $y$ from $A-A$, we exclude $2\lambda_n+1$ elements from $D(n+1)$: the first $2\lambda_n$ elements are precisely of the form $p+xk^n$ and $zk^n-p$ for some $p\in \left[k^n - 1\right] \setminus D(n)$, and the last element is $yk^n$ itself. So every gap in $A-A$ excludes $2\lambda_n+1$ elements from $D(n+1)$ independently. By supposition we also have $(k-1)k^n \in D(n+1)$, but we exclude all of the $\lambda_n$ elements of the form $(k-1)k^n + p$ for $p\in \left[k^n - 1\right] \setminus D(n)$. Hence, 
\[r_{n+1} =g(2\lambda_n+1)+\lambda_n = (2g+1)\lambda_n+g.\]
From Lemma \ref{L:incompletediff} we have $2g+1 \le \frac{k}{2}$ and we know that $g\le \lambda_n$, yielding: 
\begin{align*}
\frac{r_{n+1}}{\big|[k^{n+1}-1]\big|} = \frac{(2g+1)\lambda_n+g}{k\big|[k^{n}-1]\big|} \le \frac{\frac{k}{2}+1}{k} \frac{\lambda_n}{\big|[k^{n}-1]\big|}.
\end{align*}
So 
\[\frac{r_{n}}{\big|[k^n-1]\big|} \le \left(\frac{\frac{k}{2}+1}{k}\right)^{n-1} \frac{g}{k}.\]
When $k\ge 3$, this implies that $\lim_{n\rightarrow \infty} \frac{r_{n}}{\left|[k^n-1\right|} = 0$.

For the final part of the proof, let $\beta' = \max B$. Because $\beta'$ and $v$ are fixed, we have 
\[\lim_{n\rightarrow \infty} \frac{\frac{1}{2v^2}\flr{\frac{k^n}{\beta'}}}{k^n}= \frac{1}{2v^2\beta'}>0.\]
Thus, we can choose an $n\in \N$ such that $\frac{\lambda_n}{k^n}=\frac{\lambda_n}{\left|[k^n-1]\right|} <\frac{\frac{1}{2v^2}\flr{\frac{k^n}{\beta'}}}{k^n}$, which implies $\lambda_n < \frac{1}{2v^2}\flr{\frac{k^n}{\beta'}}$. Set $m\in \N$ such that $k^{m-1}\le \flr{\frac{k^n}{\beta'}}< k^m$, and let $S = \left[\flr{\frac{k^n}{\beta'}}\right] \cap (k^m-D(m))$ (where $k^m - D(m) = \{k^m-x:\, x\in D(m)\}$). By Corollaries \ref{C:lowerhalf} and \ref{C:evenodd}, deduce that 
\begin{equation}
|S|=\left| \left[\flr{\frac{k^n}{\beta'}}\right]\cap (k^m-D(m))\right| \ge \frac{1}{2}\left|\left[\flr{\frac{k^n}{\beta'}}\right]\right| - 1 =\frac{1}{2}\flr{\frac{k^n}{\beta'}} .\label{E:setinequality}
\end{equation}
For each element $x\in S$, let $Q_x = \{k^n-x\beta_q:\, 1\le q \le v\}$. Note that $Q_x \subset D(n)$ for every setting of $x$. Also, each set $Q_x$ intersects with at most $v(v-1)<v^2$ other sets $Q_{x'}$: those that set their smallest element equal to the second smallest of $Q_x$, their smallest equal to the third smallest of $Q_x$ (etc...), those that set their second smallest element equal to the smallest of $Q_x$, those that set their second smallest element equal to the third smallest of $Q_x$, etc... Hence, there are at least $\frac{|S|}{v^2}$ pairwise disjoint sets of the form $Q_x,\, x \in S$. But by (\ref{E:setinequality}),
\[\frac{|S|}{v^2}\ge \frac{1}{2v^2} \flr{\frac{k^n}{\beta'}}>\lambda_n.\]
Now $\lambda_n$ measures the size of $\left[ k^n -1\right] \setminus D(n)$. It follows that we must have some $x\in S$ such that $Q_x \subset D(n)$. By the definition of $S$, the statement of the lemma holds with $\gamma = x$.
\end{proof}

\begin{proposition} \label{P:tqtwo} Let $T$ be a type $(t,q)$-Chac\'{o}n Transformation with $k \ge 3$, for which there exist integers $f,g,\, 0 \le f , g < t$ with $\phi(f)-\phi(g) = 1$ (i.e.\ we refrain from adding a spacer tower on top of one of the subcolumns at each stage). Fix any $i\in \N$, and let $I$ denote the base of column $C_i$. Let $\alpha = (\alpha_1,...,\alpha_{v-1})$ be any tuple of nonzero integers, and let $b_0,...,b_{v-1}$ be any $v$-tuple of integers in $\{0,...,h_n\}^v$. Then there exists an integer $z > 0$ such that, for every $n \ge i$ and $q \in \{0,...,v-1\}$, there is a sequence of height set elements $x_{q,0} \in H_{n}$, $x_{q,1} \in H_{n+1},...,\, x_{q,z-1} \in H_{n + z -1}$ such that for any descendant tuple $(a_0,...,a_{v-1}) \in D(I,N)^v$ with summand components $a_{q,\ell + n} = x_{q,\ell}$ (assume that $N \ge n + z$) for all $q$ and $\ell \in \{0,...,z-1\}$, we can associate a unique complementary tuple $(d_0,...,d_{v-1}) \in D(I,N)^v$ satisfying 
\begin{equation} \label{E:pwmsimple}
a_0 - d_0 - b_0 = \frac{a_q - d_q - b_q}{\alpha_q} 
\end{equation}
for $ q \in \{1,...,v-1\}$. 
\end{proposition}
\begin{proof}
We will argue by Proposition \ref{P:arbproduct} that $T\times T^{\alpha} = T\times T^{\alpha_1} \times ... \times T^{\alpha_{v-1}}$ is ergodic for any tuple of nonzero integers $\alpha = (\alpha_1,...,\alpha_{v-1})$. Fix such an $\alpha$, and suppose that $T$ is as specified. Then, let $n$ be any integer that is at least $i$. For brevity, let $D(I,j)$ denote the descendants of $I$ in the $(n+j)^\text{th}$ column. Also, let $K_j = \sum_{\ell=0}^{j-1} k^\ell$ for any natural number $j$, and $K_0 = 0$. Recall that we have $H_n = \big\{0,\, \phi(1) h_n , \, \phi(2) h_n , ..., \, \phi(t-1) h_n\big\}$ and for any $j\in \N$, 
\begin{equation}
H_{n+j} = \Big\{0,\, k^j \phi(1) h_n + \phi(1) K_j,\, k^j \phi(2) h_n + \phi(2) K_j,...,\, k^j(k-1)h_n + (k-1) K_j\Big\}, \label{E:htsetstructure}
\end{equation}
where $h_n$ is the height of column $C_n$. For now, let $H_{n+j}'$ denote the set of coefficients $0,\phi(1),...,\phi(k-1)$ on $h_n$ in $H_{n+j}$, and let $D(I,j)' = \sum_{\ell=0}^{j-1} H_{n + \ell}'$. Then for any $j\in \N$, we can write 
\begin{align*}
H_{n+j}'-H_{n+j}' &= k^j \left(\Big\{0,\phi(1),\phi(2),...,k-1\Big\} - \Big\{0,\phi(1),\phi(2),...,k-1\Big\} \right)\\
& = k^j (H_n'-H_n').
\end{align*}
Also, observe that
\begin{align}
(H_n'-H_n')+k(H_n'-H_n')+...+k^{j-1}(H_n'-H_n')& = \sum_{\ell=0}^{j-1} k^\ell H_{n}'-\sum_{\ell=0}^{j-1} k^\ell H_{n}' \nonumber \\
& = \sum_{\ell=0}^{j-1} H_{n + \ell}'-\sum_{\ell=0}^{j-1} H_{n + \ell}' \nonumber \\
& = D(I,j)'-D(I,j)'. \label{E:tqdescendants}
\end{align}

By assumption and the definition of $(t,q)$-Chac\'{o}n type maps, the set $\left\{0,\phi(1),...,k-1\right\}$ satisfies the conditions of $A$ in Lemma \ref{L:incompletediff} and hence we can apply Lemma \ref{L:incompletediff2} with $B=\{|\alpha_1|,|\alpha_3|,...,|\alpha_{v-1}|\}$. Therefore, for some $m,j\in \N$, we can identify an element $\gamma \in k^m -D(I,m)'$ with $k^j-\gamma |\alpha_q| \in D(I,j)'$ for every $q\in \{1,...,v-1\}$. Specifically, for every such $q$, we have elements $x_{q,\ell}$ and $y_{q,\ell}$ in $\{0,...,t-1\}$ which satisfy:
\begin{align}
-(k^j-\gamma |\alpha_q|)=\sum_{\ell = 0}^{j-1} k^\ell\left(\phi\left(x_{q,\ell}\right) - \phi\left(y_{q,\ell}\right)\right), \label{E:tqcoeffs1}
\end{align}
where have used the fact that the difference set is symmetric about zero. Furthermore, Lemma \ref{L:incompletediff2} gives elements $x_{0,\ell}$ and $y_{0,\ell}$ in $\{0,...,t-1\}$ with 
\begin{align}
-(k^m - \gamma) = \sum_{\ell = 0}^{m-1} k^\ell\left(\phi\left(x_{1,\ell}\right) -\phi\left(y_{1,\ell}\right)\right). \label{E:tqcoeffs2}
\end{align}

For brevity, we let $a_{q,\ell}\in H_{n+\ell}$ denote the $H_{n+\ell}$-summand component of a descendant $a_q\in D(I,N)$ for some sufficiently large $N$ (i.e. $N > n + \ell$). We will do the same for elements of the form $d_{q,\ell}$.

Consider first the case when $\alpha_q$ is positive for $q\ne 0$. We appeal to equation (\ref{E:htsetstructure}) to deduce that we can set $a_{q,\ell} = k^\ell\phi(x_{q,\ell})h_n+\phi(x_{q,\ell})K_\ell$ for elements $x_{q,\ell},\, y_{q,\ell} \in \{0,...,t-1\}$. Similarly, set $d_{q,\ell} = k^\ell\phi(y_{q,\ell})h_n+\phi(y_{q,\ell})K_\ell$. Hence, from $\ell = 0 $ to $\ell = j-1$, equation (\ref{E:tqcoeffs1}) implies that
\begin{align}
\sum_{\ell=0}^{j-1} a_{q,\ell}-d_{q,\ell} & = \sum_{\ell=0}^{j-1}\left(k^\ell\phi(x_{q,\ell})h_n+\phi(x_{q,\ell})K_\ell-k^\ell\phi(y_{q,\ell})h_n-\phi(y_{q,\ell})K_\ell\right) \nonumber \\
& = -\left(k^j-\gamma \alpha_i\right)h_n + \sum_{\ell=0}^{j-1}K_\ell\left(\phi(x_{q,\ell})-\phi(y_{q,\ell})\right). \label{E:posalpha}
\end{align}
Similarly, 
\begin{align}
\sum_{\ell=0}^{m-1} a_{1,\ell}-d_{1,\ell} & = -\left(k^m-\gamma\right)h_n + \sum_{\ell=0}^{m-1}K_\ell\left(\phi(x_{1,\ell})-\phi(y_{1,\ell})\right). \label{E:alphaone}
\end{align}
For values of $q$ with $\alpha_q<0$, we invert the output of (\ref{E:posalpha}) by reversing the assignments of $a_{q,\ell}$ and $d_{q,\ell}$ to obtain
\begin{align}
\sum_{\ell=0}^{j-1} a_{q,\ell}-d_{q,\ell} & = \left(k^j-\gamma |\alpha_q|\right)h_n - \sum_{\ell=0}^{j-1}K_\ell\left(\phi(x_{q,\ell})-\phi(y_{q,\ell})\right). \label{E:negalpha}
\end{align}
For a natural number $r_q$, consider what happens in equation (\ref{E:posalpha}) when we also set $a_{q,\ell}=0$ and $d_{q,\ell}=k^\ell(k-1)h_n + (k-1)K_\ell$ (the maximal element of $H_{n + \ell}$) for $j \le \ell < j+r_q$, and $a_{q,j+r_q}-d_{q,j+r_q}=k^{j+r_q}h_n+K_{j+r_q}$, which we can do by supposition on $T$. Then we should have 
\begin{align}
\sum_{\ell=0}^{j+r_q} a_{q,\ell}-d_{q,\ell} =& -\left(k^j-\gamma \alpha_q\right)h_n + \sum_{\ell=0}^{j-1}K_\ell\left(\phi(x_{q,\ell})-\phi(y_{q,\ell})\right)+k^{j+r_q}h_n+K_{j+r_q}\nonumber\\
&-\sum_{\ell=j}^{j+r_q-1} \left(k^\ell(k-1)h_n + (k-1)K_\ell\right)\nonumber\\
= & \gamma \alpha_q h_n + \sum_{\ell= 0 }^{j-1} K_\ell \left( k-1+\phi(x_{q,\ell})-\phi(y_{q,\ell}) \right) - \sum_{\ell = 0}^{j+r_q-1}(k-1)K_\ell + K_{j+r_q} \nonumber \\
= & \gamma \alpha_q h_n + \sum_{\ell= 0 }^{j-1} K_\ell \left( k-1+\phi(x_{q,\ell})-\phi(y_{q,\ell}) \right) + j+r_q, \label{E:posalpha2}
\end{align}
where simplification has resulted from the definition of $K_\ell$ and the fact that $\sum_{\ell = j}^{j + z - 1} k^\ell (k-1) = k^{j + z } - k^j$ for any $ z > 0$. In a similar fashion, but by reversing the assignments of $a_{q,\ell}$ and $d_{q,\ell}$, for values of $q$ with $\alpha_q<0$ we could extend (\ref{E:negalpha}) to the following: 
\begin{align}
\sum_{\ell=0}^{j+r_q} a_{q,\ell}-d_{q,\ell} =& \left(k^j+\gamma \alpha_q\right)h_n - \sum_{\ell=0}^{j-1}K_\ell\left(\phi(x_{q,\ell})-\phi(y_{q,\ell})\right)-k^{j+r_q}h_n-K_{j+r_q}\nonumber\\
&+\sum_{\ell=j}^{j+r_q-1} \left(k^\ell(k-1)h_n + (k-1)K_\ell\right)\nonumber\\
=& \gamma \alpha_q h_n -\sum_{\ell= 0 }^{j-1} K_\ell \left( k-1+\phi(x_{q,\ell})-\phi(y_{q,\ell}) \right) - j-r_q. \label{E:negalpha2}
\end{align}
And similarly, we can specify $a_{0,\ell}$ and $d_{0,\ell}$ such that 
\begin{align}
\sum_{\ell=0}^{m+r_1} a_{1,\ell}-d_{1,\ell} = & \gamma h_n + \sum_{\ell= 0 }^{m-1} K_\ell \left( k-1+\phi(x_{1,\ell})-\phi(y_{1,\ell}) \right) + m+r_1. \label{E:alphaone2}
\end{align}
This brings us to the final stage. In equations (\ref{E:posalpha2}) and (\ref{E:negalpha2}), let 
\[L_q = \sum_{\ell= 0 }^{j-1} K_\ell \left( k-1+\phi(x_{q,\ell})-\phi(y_{q,\ell}) \right) + j,\]
and in equation (\ref{E:alphaone2}), let 
\[L_0 = \sum_{\ell= 0 }^{m-1} K_\ell \left( k-1+\phi(x_{1,\ell})-\phi(y_{1,\ell}) \right) + m.\]
Let $\{b_0,...,b_{v-1}\}\subset \N\cup \{0\}$. Pick $r_0$ high enough such that for all nonzero $q$ with $\alpha_q > 0$:
\[\alpha_q(L_0+r_0-b_0)+b_q>L_q,\]
and for all $q$ with $\alpha_q<0$:
\[-\alpha_q(L_0+r_0-b_0)-b_q>L_q.\]
In both cases, for $q \ne 0$ we will set $r_q = |\alpha_q|(L_0+r_0-b_0)+\text{sgn}(\alpha_q)b_q-L_q$. Then by letting $a_{q,\ell}$ and $d_{q,\ell}$ be as specified for $\ell = 0 $ to $\ell = j+r_q$ (note that $d_{q,\ell}$ is unique to $a_{q,\ell}$ in this selection), and setting $d_{q,\ell}=a_{q,\ell}$ elsewhere, we obtain descendants $a_0,...,a_{v-1},d_0,...,d_{v-1}$ such that for nonzero $q$ with $\alpha_q>0$,
\begin{align*}
a_q-d_q &= \gamma \alpha_q h_n+ L_q+r_q \\
&= \gamma \alpha_q h_n + L_q + \alpha_q (L_0+r_0-b_0)+b_q-L_q\\
&= \alpha_q(\gamma h_n + L_0+r_0-b_0)+b_q\\
& = \alpha_q (a_0-d_0-b_0)+b_q,
\end{align*}
by (\ref{E:posalpha2}) and similarly, for $q$ with $\alpha_q<0$,
\begin{align*}
a_q-d_q &= \gamma \alpha_q h_n- L_q-r_q \\
&= \gamma \alpha_q h_n - L_q + \alpha_q (L_0+r_0-b_0)+b_q+L_q\\
&= \alpha_q(\gamma h_n + L_0+r_0-b_0)+b_q\\
& = \alpha_q (a_0-d_0-b_0)+b_q.
\end{align*}
by (\ref{E:negalpha2}). Hence, by fixing at most $j+\max\{r_q\}$ summand components of $a_0,...,a_{v-1}$, we are able to identify a unique set of descendants $d_0,...,d_{v-1}$  such that for every $q\in \{1,...,v-1\}$, the power weakly mixing condition described in (\ref{E:pwmsimple}) is fulfilled. So the proof works with $z = j+\max\{r_q\}$. 
\end{proof}

\begin{theorem} \label{T:tqfinal}
A $(t,q)$-type Chac\'{o}n map is power weakly mixing if and only if we omit a spacer stack from at least one subcolumn that is not the rightmost subcolumn. 
\end{theorem}
\begin{proof}
First we show sufficiency. If $k < 3$, then clearly $k =2$, which is impossible if $T$ is an infinite measure preserving transformation. Thus, if we omit a spacer stack from at least one subcolumn that is not the rightmost one, Proposition \ref{P:tqtwo} applies. Let $i$ be any positive integer, $I$ the base of $C_i$, $\alpha$ any $(v-1)$-tuple of nonzero integers, and $b_0,...,b_{v-1}$ be any tuple in $\big\{0,...,h_i - 1\big\}^v$. By Proposition \ref{P:tqtwo}, there exists an integer $z\equiv z(b_0,\ldots , b_{v-1})$ such that, for $N \ge i + z$, any descendant tuple $(a_0,...,a_{v-1}) \in D(I,N)$ with elements having a certain summand sequence $a_{q,i},...,a_{q,i+z -1}$ for all $q \in \{0,...,v-1\}$ can be associated to a unique descendant tuple $(d_0,...,d_{v-1})$ satisfying the power weakly mixing condition in (\ref{E:pwmsimple}). Hence, Proposition \ref{P:arbproduct} implies that $T \times T^{\alpha_1} \times \cdots \times T^{\alpha_{v-1}}$ is ergodic with $\beta(b_0,\ldots , b_{v-1}) = \frac{1}{t^{vz}}$. Because $(\alpha_1,...,\alpha_{v-1})$ was arbitrary, $T$ is power weakly mixing. 

To show necessity, consider a $(t,q)$-type map $T$ with spacer stacks added to every subcolumn except for the last one. Then for any $n\in \N$, we have a height set $H_n$ of the following form:
\[H_n = \big\{0,2h_n,4h_n,...,(k-1)h_n\big\}.\]
Let $I$ denote the base of the first column $C_0$. Inductively, we can show that $h_n\ge \max D(I,n)+2$ for every $n\in \N$. In the base case with $n=1$, we have $h_1 = kh_0+1 = k+1$, and $\max D(I,1) = (k-1)h_0 = k-1$. Suppose that for some $n\in \N$, we have $h_n \ge \max D(I,n)+2$. Then 
\begin{align*}
h_{n+1} =& kh_n +1 \\
=& (k-1)h_n + h_n + 1 \\
\ge & (k-1) h_n + \max D(I,n)+2\\
= & \max D(I,n+1) + 2. 
\end{align*}
Suppose that $T\times T$ was ergodic; then for some $N\in \N$, we should be able to find elements $a_1,a_2,d_1,d_2\in D(I,N)$ such that 
\begin{align}\label{E:ergnecess}
 \sum_{i=0}^{N-1}\left( a_{1,i}-a_{2,i} \right) = a_1-a_2 = d_1-d_2-1=\sum_{i=0}^{N-1}\left( d_{1,i}-d_{2,i} \right) - 1 . 
 \end{align}
Let $m\in \{0,...,N-1\}$ be the highest natural number such that $a_{1,m}-a_{2,m} \ne d_{1,m}-d_{2,m}$ (such an $m$ clearly must exist). Then 
\[\left|a_{1,m}-a_{2,m} -( d_{1,m}-d_{2,m})\right| \ge 2h_m \ge 2(\max(D(I,m)+2),\]
implying the following inequality: 
\begin{align*}
\Big|a_1-a_2-(d_1-d_2)\Big| & = \left| a_{1,m}-a_{2,m}-d_{1,m}+d_{2,m}+\sum_{i=0}^{m-1}\left( a_{1,i}-a_{2,i}-d_{1,i}+d_{2,i}\right)\right|\\
& \ge 4,
\end{align*}
which contradicts \eqref{E:ergnecess}.
\end{proof}


\bibliographystyle{plain}
\bibliography{ergodicOct21}

\begin{thebibliography}{10}

\bibitem{AdFrSi97}
Terrence Adams, Nathaniel Friedman, and Cesar~E. Silva.
\newblock Rank-one weak mixing for nonsingular transformations.
\newblock {\em Israel J. Math.}, 102:269--281, 1997.

\bibitem{AdFrSi01}
Terrence Adams, Nathaniel Friedman, and Cesar~E. Silva.
\newblock Rank-one power weakly mixing non-singular transformations.
\newblock {\em Ergodic Theory Dynam. Systems}, 21(5):1321--1332, 2001.

\bibitem{AdSi16}
Terrence~M. Adams and Cesar~E. Silva.
\newblock Weak rational ergodicity does not imply rational ergodicity.
\newblock {\em Israel J. Math.}, 214(1):491--506, 2016.

\bibitem{CFKLPS}
Julien Clancy, Rina Friedberg, Indraneel Kasmalkar, Isaac Loh, Tudor
  P{\u{a}}durariu, Cesar~E. Silva, and Sahana Vasudevan.
\newblock Ergodicity and conservativity of products of infinite transformations
  and their inverses.
\newblock {\em Colloq. Math.}, 143(2):271--291, 2016.

\bibitem{Da16b}
Alexandre~I. Danilenko.
\newblock Infinite measure~preserving~transformations with strong {R}adon
  {MSJ}.
\newblock {\em https://arxiv.org/abs/1610.07285}.

\bibitem{Da04}
Alexandre~I. Danilenko.
\newblock Infinite rank one actions and nonsingular {C}hacon transformations.
\newblock {\em Illinois J. Math.}, 48(3):769--786, 2004.

\bibitem{Da16}
Alexandre~I. Danilenko.
\newblock Finite ergodic index and asymmetry for infinite measure preserving
  actions.
\newblock {\em Proc. Amer. Math. Soc.}, 144(6):2521--2532, 2016.

\bibitem{DGMS99}
Sarah~L. Day, Brian~R. Grivna, Earle~P. McCartney, and Cesar~E. Silva.
\newblock Power weakly mixing infinite transformations.
\newblock {\em New York J. Math.}, 5:17--24 (electronic), 1999.

\bibitem{GrHiWa03}
K.~Gruher, F.~Hines, D.~Patel, C.~E. Silva, and R.~Waelder.
\newblock Power weak mixing does not imply multiple recurrence in infinite
  measure and other counterexamples.
\newblock {\em New York J. Math.}, 9:1--22 (electronic), 2003.

\bibitem{JRR}
{\'E}.~Janvresse, Emmanuel Roy, and T.~de~la Rue.
\newblock Invariant measures for cartesian powers of chacon infinite
  transformation, https://arxiv.org/abs/1505.08033.

\bibitem{KaPa63}
S.~Kakutani and W.~Parry.
\newblock Infinite measure preserving transformations with ``mixing''.
\newblock {\em Bull. Amer. Math. Soc.}, 69:752--756, 1963.

\bibitem{LoSi17}
Isaac Loh and Cesar~E. Silva.
\newblock Strict doubly ergodic infinite transformations.
\newblock {\em Dynamical Systems}, 0(ja):1--24, 0.

\bibitem{AM}
Abhaya Menon.
\newblock Power weak mixing and recurrence in infinite measure transformations.
\newblock {\em Thesis, Williams College}, 2001.

\bibitem{Si08}
C.~E. Silva.
\newblock {\em Invitation to ergodic theory}, volume~42 of {\em Student
  Mathematical Library}.
\newblock American Mathematical Society, Providence, RI, 2008.

\end{thebibliography}

\end{document}